\documentclass[12pt, reqno]{amsart}
\usepackage{amsmath, amsthm, amscd, amsfonts, amssymb}
\usepackage[bookmarksnumbered]{hyperref}

\newtheorem{theorem}{Theorem}[section]
\newtheorem{lemma}[theorem]{Lemma}

\newtheorem{corollary}[theorem]{Corollary}
\theoremstyle{definition}
\newtheorem{definition}[theorem]{Definition}

\theoremstyle{remark}

\numberwithin{equation}{section}

\begin{document}

\title[]{Nearly generalized Jordan derivations }
\author[M. Eshaghi Gordji and N. Ghobadipour]{M. Eshaghi Gordji and N. Ghobadipour}
\address{Department of Mathematics, Semnan University, P. O. Box 35195-363, Semnan, Iran}
\email{madjid.eshaghi@gmail.com  and  ghobadipour.n@gmail.com}



\begin{abstract} Let $A$ be an algebra and let $X$ be an $A$-bimodule. A $\Bbb C-$linear
mapping $d:A \to X$ is called a generalized Jordan derivation if
there exists a Jordan derivation (in the usual sense) $\delta:A \to
X$ such that $d(a^2)=ad(a)+\delta(a)a$ for all $a \in A.$ The main
purpose of this paper to prove the  Hyers-Ulam-Rassias
 stability and superstability of the generalized
Jordan derivations.
\end{abstract}
\maketitle

%--------------------------------------------------------------------------------------%

\section{Introduction }
We say a functional equation $(\xi)$ is stable if any function $g$
satisfying the equation $(\xi)$ approximately is near to a true
solution of $(\xi).$ The equation $(\xi)$ is called superstable if
every approximate solution of $(\xi)$ is an exact solution. It
seems that the stability problem of functional equations had been
first raised by Ulam (cf. \cite{Ul}):  Let $(G_1,.)$ be a group
and let $(G_2,*)$ be a metric group with the metric $d(.,.).$
Given $\epsilon >0$, dose there exist a $\delta
>0$, such that if a mapping $h:G_1\longrightarrow G_2$ satisfies the
inequality $d(h(x.y),h(x)*h(y)) <\delta$ for all $x,y\in G_1$,
then there exists a homomorphism $H:G_1\longrightarrow G_2$ with
$d(h(x),H(x))<\epsilon$ for all $x\in G_1?$ In the other words,
Under what condition dose there exists a homomorphism near an
approximate homomorphism? The concept of stability for functional
equation arises when we replace the functional equation by an
inequality which acts as a perturbation of the equation. In 1941,
D. H. Hyers \cite{Hy} gave a first affirmative  answer to the
question of Ulam for Banach spaces. Let $f:{E}\longrightarrow{E'}$
be a mapping between Banach spaces such that
$$\|f(x+y)-f(x)-f(y)\|\leq \delta $$
for all $x,y\in E,$ and for some $\delta>0.$ Then there exists a
unique additive mapping $T:{E}\longrightarrow{E'}$ such that
$$\|f(x)-T(x)\|\leq \delta$$
for all $x\in E.$ Now assume that $E$ and $E^{'}$ are real normed
spaces with $E^{'}$ complete, $f:E \to E^{'}$ is a mapping such
that for each fixed $x \in E$ the mapping $t\rightarrowtail f(tx)$
is continuous on $\Bbb R,$ and that there exist $\delta\geq 0$ and
$p\neq 1$ such that $$\|f(x+y)-f(x)-f(y)\|\leq \delta
(\|x\|^p+\|y\|^p)$$ for all $x,y \in E.$ It was shown by Rassias
\cite{Ra} for $p\in[0,1)~$ (and indeed $p<1$) and Gajda \cite{Gaj}
following the same approach as in \cite{Ra} for $p>1$ that there
exists a unique linear map $T:E \to E^{'}$ such that
$$\|f(x)-T(x)\| \leq \frac{2 \delta \|x\|^p}{|2^p-2|}$$ for all $x
\in E.$ This phenomenon is called Hyers-Ulam-Rassias stability. it
is shown that there is no
analogue of Rassias' result for $p=1$ (see \cite{Gaj,Ra}).\\
In 1994, a generalization of the Rassias' theorem was obtained by
G\v avruta as follows \cite{G}.\\
Suppose (G,+) is an abelian group, $E$ is a Banach space, and that
the so-called admissible control function $\varphi:G \times G \to
\Bbb R$ satisfies
$$\tilde{\varphi}(x,y):=2^{-1}\sum_{n=0}^{\infty}2^{-n}\varphi(2^nx,2^ny)<
\infty$$ for all $x,y \in G.$ If $f:G \to E$ is a mapping with
$$\|f(x+y)-f(x)-f(y)\|\leq \varphi(x,y)$$ for all $x,y \in G,$
then there exists a unique mapping $T:G \to E$ such that
$T(x+y)=T(x)+T(y)$ and $\|f(x)-T(x)\| \leq \tilde{\varphi}(x,x)$
for all $x,y \in G.$\\
Since then several stability problems of various functional
equations have been investigated by many mathematicians. The
reader is referred to \cite{Cz,Ra} for a comprehensive account of
the subject.\\
Generalized derivations and generalized  Jordan derivations first
appeared in the context of operator algebras \cite{Ma}. Later,
these were introduced in the framework of
pure algebra \cite{FZ,HV}.\\
\begin{definition}
Let $A$ be an algebra and let $X$ be an $A$-bimodule. A linear
mapping $d:A \to X$ is called a generalized derivation if there
exists a derivation (in the usual sense) $\delta:A \to X$ such
that $d(ab)=ad(b)+\delta(a)b$ for all $a,b \in A.$\\ Every right
multiplier is a generalized derivation.
\end{definition}
\begin{definition}\cite{FZ}
Let $A$ be an algebra and let $X$ be an $A$-bimodule. A linear
mapping $d:A \to X$ is called a generalized Jordan derivation if
there exists a Jordan derivation (in the usual sense) $\delta:A
\to X$ such that $d(a^2)=ad(a)+\delta(a)a$ for all $a \in A.$\\
The stability of derivations was studied by Park in \cite{Pa}. M.
Moslehian \cite{Mo} investigated the Hyers-Ulam-Rassias stability
of generalized derivations from a unital normed algebra A to a unit linked Banach  A-bimodule.\\
In this paper, we investigate the  Hyers-Ulam-Rassias stability and
moreover prove superstability of generalized Jordan derivations.

\end{definition}
\section{Main result}
In this section, we investigate Hyers-Ulam-Rassias stability of
generalized Jordan derivations from a unital Banach algebra to a
unit linked Banach  A-bimodule and use some ideas of \cite{Mo,Pa}.\\
Throughout this section, assume that $A$ is a unital Banach
algebra and, let $X$ be unit linked Banach  A-bimodule.\\
We need the following lemma in the main results of the present
paper.
\begin{lemma}\cite{Par}
Let $U,V$ be linear spaces and let $f:U \to V$ be an additive
mapping such that $f(\lambda x)=\lambda f(x)$ for all $x \in U$
and all $\lambda \in \Bbb T^1:=\{\lambda \in \Bbb C~;
|\lambda|=1\}.$ Then the mapping $f$ is $\Bbb C$-linear.
\end{lemma}
We start our work with a result concerning the superstability of
the generalized Jordan derivations as follows.
\begin{theorem}
Let $p<1$ and $\theta$ be nonnegative real numbers. Suppose $f:A
\to X$ is a mapping with $f(0)=0$ for which there exists a map
$g:A \to X$ with $g(0)=g(1)=0$ such that $$\|f(a+\lambda
b+c^2)-f(a)-\lambda f(b)-cf(c)-g(c)c\| \leq \theta \|f(c)\|,
\eqno(2.1)$$
$$\|g(\lambda ab +\lambda c)-\lambda a g(b)-\lambda g(a)b-\lambda g(c)\| \leq \theta (\|a\|^p+\|b\|^p+\|c\|^p) \eqno(2.2)$$ for all
$a,b,c \in A$ and all $\lambda \in \Bbb T^{1}.$ Then $f:A \to X$
is a generalized Jordan derivation.
\end{theorem}
\begin{proof}
Letting $c=0$ and $\lambda=1$ in $(2.1),$ we get
$$f(a+b)=f(a)+f(b)$$ for all $a,b \in A.$ So $f$ is additive.\\
Letting $a=c=0$ in $(2.1),$ we get $f(\lambda b)=\lambda f(b)$ for
all $b \in A$ and all $\lambda \in \Bbb T^{1}.$ By Lemma $2.1,$
the mapping $f$ is  $\Bbb C$-linear.\\
Putting $a=b=0$ and $\lambda=1$ in $(2.1),$ we get
$$\|f(c^2)-cf(c)-g(c)c\| \leq \theta \|f(c)\| \eqno(2.3)$$ for all $c \in A.$
Replacing $c$ by $2^nc$ in $(2.3)$ we obtain
$$\|f(2^{2n}c^2)-2^ncf(2^nc)-2^ng(2^nc)c\| \leq  \theta \|f(2^nc)\|,$$
whence $$\|2^{-2n}f(2^{2n}c^2)-2^{-n}cf(2^nc)-2^{-n}g(2^nc)c\|
\leq 2^{-2n} \theta \|f(2^nc)\|$$ for all $c \in A.$ Hence
$$\|f(c^2)-cf(c)-2^{-n}g(2^nc)c\| \leq
2^{-n} \theta \|f(c)\| \eqno(2.4)$$ for all $c \in A.$\\
Let $n$ tend to $\infty$ in $(2.4).$ Then
$$f(c^2)=cf(c)+\lim_{n \to \infty}2^{-n}g(2^nc)c$$ for all $c \in
A.$ By Hyers' Theorem, the sequence $\{2^{-n}g(2^nc)\}$ is
convergent. Set $\delta(c):=\lim_{n \to \infty}2^{-n}g(2^nc)$ for
all $c \in A.$ Hence $$f(c^2)=cf(c)+\delta(c)c \eqno(2.5)$$ for
all $c \in A.$\\ Next we claim that $\delta$ is a Jordan
derivation. Putting $b=1$ and replacing $a,c$ by $2^na,2^nc,$
respectively, in $(2.2),$ we get $$ \|g(2^n(\lambda a +\lambda
c))-\lambda g(2^na)-\lambda g(2^nc)\|\leq
\theta(\|2^na\|^p+\|2^nc\|^p+1),$$ whence $$2^{-n}\|g(2^n(\lambda
a +\lambda c))-\lambda g(2^na)-\lambda g(2^nc)\|\leq
2^{-n}\theta(\|2^na\|^p+\|2^nc\|^p+1) \eqno(2.6)$$ for all $a,c
\in A$ and all $\lambda \in \Bbb T^{1}.$ Let $n$ tend to $\infty$
in $(2.6).$ Then $$\delta(\lambda a +\lambda c)=\lambda
\delta(a)+\lambda \delta(c) \eqno(2.7)$$ for all $a,c \in A$ and
all $\lambda \in \Bbb T^{1}.$ Hence by Lemma $2.1$ $\delta$ is
$\Bbb C$-linear. Now, letting $c=0$, $\lambda=1$ and replacing $b$
by $a$ in $(2.2),$ we get
$$\|g(a^2)-ag(a)-g(a)a\|\leq 2\theta \|a\|^p$$ for all $a \in A.$
Replacing $a$ by $2^na$ in above inequality, we get
$$\|g(2^{2n}a^2)-2^nag(2^na)-2^ng(2^na)a\|\leq 2\theta \|2^na\|^p
\eqno(2.8)$$ for all $a\in A.$ Thus
$$\|2^{-2n}g(2^{2n}a^2)-2^{-n}ag(2^na)-2^{-n}g(2^na)a\|\leq \theta
2^{n(p-2)} \|a\|^p\eqno(2.9)$$ for all $a\in A.$ Hence by letting
$n \to \infty$ in $(2.9),$ we conclude that
$\delta(a^2)=a\delta(a)+\delta(a)a$ for all $a \in A.$ It then
follows from $(2.5)$ that $f$ is a generalized Jordan derivation.
\end{proof}
\begin{theorem}
Let $p>1$ and $\theta$ be nonnegative real numbers. Suppose $f:A
\to X$ is a mapping with $f(0)=0$ for which there exists a map
$g:A \to X$ with $g(0)=g(1)=0$ such that $$\|f(a+\lambda
b+c^2)-f(a)-\lambda f(b)-cf(c)-g(c)c\| \leq \theta \|f(c)\|, $$
$$\|g(\lambda ab +\lambda c)-\lambda a g(b)-\lambda g(a)b-\lambda g(c)\| \leq \theta (\|a\|^p+\|b\|^p+\|c\|^p) $$ for all
$a,b,c \in A$ and all $\lambda \in \Bbb T^{1}.$ Then $f:A \to X$
is a generalized Jordan derivation.
\end{theorem}
\begin{proof}
The proof is similar to the proof of Theorem $2.2.$
\end{proof}
Now we prove the  Hyers-Ulam-Rassias stability of generalized Jordan
derivations.
\begin{theorem}
Suppose $f:A \to X$ is a mapping with $f(0)=0$ for which there
exist a map $g:A \to X$ with $g(0)=g(1)=0$ and a function
$\varphi:A\times A \times A \to \Bbb R^+$ such that
\begin{align*}
\max\{&\|f(\lambda a+\lambda b+c^2)-\lambda f(a)-\lambda
f(b)-cf(c)-g(c)c\|, \\
&\|g(\lambda ab +\lambda c)-\lambda a
g(b)-\lambda g(a)b-\lambda g(c)\|\}
 \leq \varphi(a,b,c), \hspace {4cm}(2.10)
\end{align*}
$$\tilde{\varphi}(a,b,c):=2^{-1}\sum_{i=0}^{\infty}2^{-i}\varphi(a,b,c)< \infty \eqno(2.11)$$
for all $\lambda \in A$ and all $a,b,c \in A.$ Then there exists a
unique generalized Jordan derivation $d:A \to X$ such that
$$\|f(a)-d(a)\| \leq\tilde{\varphi}(a,a,0) \eqno(2.12)$$
for all $a \in A.$
\end{theorem}
\begin{proof}
By $(2.10)$ we have $$\|f(\lambda a+\lambda b+c^2)-\lambda
f(a)-\lambda f(b)-cf(c)-g(c)c\|\leq \varphi(a,b,c), \eqno(2.13)$$
$$\|g(\lambda ab +\lambda c)-\lambda a g(b)-\lambda g(a)b-\lambda
g(c)\|\}
 \leq \varphi(a,b,c) \eqno(2.14)$$ for all $a,b,c \in A.$ Setting
$c=0$ and $\lambda=1$ in $(2.13),$ we have
$$\|f(a+b)-f(a)-f(b)\| \leq \varphi(a,b,0) \eqno(2.15)$$ for all
$a,b \in A.$ Now we use the Rassias method on inequality $(2.15)$
(see \cite{G,Mos}). One can use induction on $n$ to show that
$$\|2^{-n}f(2^na)-f(a)\| \leq 2^{-1}\sum_{i=0}^{n-1}2^{-i}\varphi(2^ia,2^ia,0)
\eqno(2.16)$$ for all $n \in \Bbb N$ and all $a \in A,$ and that
$$\|2^{-n}f(2^na)-2^{-m}f(2^ma)\|\leq 2^{-1}\sum_{i=m}^{n-1}2^{-i}\varphi(2^ia,2^ia,0)
\eqno(2.17)$$ for all $n>m$ and all $a \in A.$ It follows from the
convergence $(2.11)$ that the sequence ${2^{-n}f(2^na)}$ is
Cauchy. Due to the completeness of $X,$ this sequence is
convergent. Set
$$d(a):=\lim_{n \to \infty}2^{-n}f(2^na). \eqno(2.18)$$
Putting $c=0$ and replacing $a,b$ by $2^na,2^nb,$ respectively, in
$(2.13),$ we get $$\|2^{-n}f(2^n(\lambda a+\lambda
b))-2^{-n}\lambda f(2^na)-2^{-n}\lambda f(2^nb)\|\leq
2^{-n}\varphi(2^na,2^nb,0) \eqno(2.19)$$ for all $a,b \in A$ and
all $\lambda \in \Bbb T^{1}.$ Taking the limit as $n \to \infty$
we obtain $$d(\lambda a+\lambda b)=\lambda d(a)+\lambda d(b)$$ for
all $a,b \in A$ and all $\lambda \in \Bbb T^{1}.$ So by Lemma
$2.1,$ the mapping $d$ is $\Bbb C$-linear.\\
Moreover, it follows from $(2.16)$ and $(2.18)$ that
$\|f(a)-d(a)\|\leq \tilde{\varphi}(a,a,0)$ for all $a \in A.$ it is
known that the additive mapping $d$ satisfying $(2.12)$ is unique
\cite{Ba}. Putting $\lambda=1,a=b=0,$ and replacing $c$ by $2^nc$ in
$(2.13),$ we get $$\|f(2^{2n}c^2)-2^ncf(2^nc)-2^ng(2^nc)c\| \leq
\varphi(0,0,2^nc), \eqno(2.20)$$ whence
$$\|2^{-2n}f(2^{2n}c^2)-2^{-n}cf(2^nc)-2^{-n}g(2^nc)c\| \leq
2^{-2n} \varphi(0,0,2^nc) \eqno(2.21)$$ for all $c \in A.$ By
$(2.18),$ $\lim_{n \to \infty}2^{-2n}f(2^{2n}a)=d(a)$ and by the
convergence of series $(2.11),$
$\lim_{2^{-2n}}\varphi(0,0,2^nc)=0.$ Hence the sequence
${2^{-n}g(2^nc)}$ is convergent. Set $\delta(c):=\lim_{n \to
\infty}2^{-n}g(2^nc)$ for all $c \in A.$ Let $n$ tend to $\infty$
in $(2.21).$ Then
$$d(c^2)=cd(c)+\delta(c)c. \eqno(2.22)$$ The rest of the proof is
similar to the proof of Theorem $2.2.$
\end{proof}
\begin{corollary}
Suppose $f:A \to X$ is a mapping with $f(0)=0$ for which there
exist constant $\theta\geq 0,~$ $p<1$ and a map $g:A \to X$ with
$g(0)=g(1)=0$ such that
\begin{align*}
\max \{&\|f(\lambda a+\lambda b+c^2)-\lambda f(a)-\lambda
f(b)-cf(c)-g(c)c\|, \\
&\|g(\lambda ab +\lambda c)-\lambda a g(b)-\lambda g(a)b-\lambda
g(c)\| \leq \theta (\|a\|^p+\|b\|^p+\|c\|^p) \hspace{2cm}(2.23)
\end{align*}
for all $\lambda \in A$ and all $a,b,c \in A.$ Then there exists a
unique generalized Jordan derivation $d:A \to X$ such that
$$\|f(a)-d(a)\| \leq \frac{\theta \|a\|^p}{1-2^{p-1}} \eqno(2.24)$$
for all $a \in A.$
\end{corollary}
\begin{proof}
It follows from Theorem $2.4$ by Putting
$\varphi(a,b,c)=\theta(\|a\|^p+\|b\|^p+\|c\|^p)$.
\end{proof}
\begin{theorem}
Suppose $f:A \to X$ is a mapping with $f(0)=0$ satisfying $(2.13)$
for which there exist a map $g:A \to X$ with $g(0)=g(1)=0$
satisfying $(2.14)$ and a function $\varphi:A\times A \times A \to
\Bbb R^+$ such that
$$\tilde{\varphi}(a,b,c):=2^{-1}\sum_{i=1}^{\infty}2^{-i}\varphi(2^{-i}a,2^{-i}b,2^{-i}c)< \infty \eqno(2.25)$$
for all $a,b,c \in A.$ Then there exists a unique generalized
Jordan derivation $d:A \to X$ such that
$$\|f(a)-d(a)\| \leq\tilde{\varphi}(a,a,0) \eqno(2.26)$$
for all $a \in A.$
\end{theorem}
\begin{proof}
The  proof is similar to the proof of Theorem 2.4.
\end{proof}
\begin{corollary}
Suppose $f:A \to X$ is a mapping with $f(0)=0$ satisfying $(2.23)$
for which there exist constant $\theta\geq 0,~$ $p>1$ and a map $g:A
\to X$ with $g(0)=g(1)=0$ satisfying $(2.23).$ Then there exists a
unique generalized Jordan derivation $d:A \to X$ such that
$$\|f(a)-d(a)\| \leq \frac{\theta \|a\|^p}{2^{p-1}-1} $$
for all $a \in A.$
\end{corollary}
\begin{proof}
It follows from Theorem 2.6 by putting
$\varphi(a,b,c)=\theta(\|a\|^p+\|b\|^p+\|c\|^p)$.
\end{proof}
\begin{corollary}
Suppose $f:A \to X$ is a mapping with $f(0)=0$ for which there
exist constant $\theta\geq 0$ and a map $g:A \to X$ with
$g(0)=g(1)=0$ such that
\begin{align*}
\max\{&\|f(\lambda a+\lambda b+c^2)-\lambda f(a)-\lambda
f(b)-cf(c)-g(c)c\|, \\
&\|g(\lambda ab +\lambda c)-\lambda a g(b)-\lambda g(a)b-\lambda
g(c)\|\} \leq \theta
\end{align*}
for all $\lambda \in A$ and all $a,b,c \in
A.$ Then there exists a unique generalized Jordan derivation $d:A
\to X$ such that
$$\|f(a)-d(a)\| \leq \theta $$
for all $a \in A.$
\end{corollary}
\begin{proof}
Letting $p=0$ in Corollary $(2.5),$ we obtain the result.
\end{proof}

\end{document}